\documentclass[11pt]{article}
\usepackage{amsmath,amsthm}
\usepackage{amssymb,mathrsfs}
\usepackage{bm,mathtools}
\usepackage{tikz}

\usepackage{xcolor}
\usepackage{geometry}
\usepackage[colorlinks=true,
linkcolor=blue,citecolor=blue,
urlcolor=blue]{hyperref}

\makeatletter
\def\@seccntDot{.}
\def\@seccntformat#1{\csname the#1\endcsname\@seccntDot\hskip 0.5em}
\renewcommand\section{\@startsection{section}{1}{\z@}%
{18\p@ \@plus 6\p@ \@minus 3\p@}%
{9\p@ \@plus 6\p@ \@minus 3\p@}%
{\large\bfseries\boldmath}}
\renewcommand\subsection{\@startsection{subsection}{2}{\z@}%
{12\p@ \@plus 6\p@ \@minus 3\p@}%
{3\p@ \@plus 6\p@ \@minus 3\p@}%
{\bfseries\boldmath}}
\renewcommand\subsubsection{\@startsection{subsubsection}{3}{\z@}%
{12\p@ \@plus 6\p@ \@minus 3\p@}%
{\p@}%
{\bfseries\boldmath}}
\makeatother

\usepackage{microtype}

\theoremstyle{plain}
\newtheorem{theorem}{Theorem}[section]
\newtheorem{lemma}{Lemma}[section]
\newtheorem{corollary}{Corollary}[section]
\newtheorem{proposition}{Proposition}[section]
\newtheorem{conjecture}{Conjecture}[section]

\theoremstyle{definition}

\numberwithin{equation}{section}
\allowdisplaybreaks
\title{On the Spread of Graph-Related Matrices\thanks{Supported by National Natural 
Science Foundation of China (12571360, 12331012, 12471320), Excellent University 
Research and Innovation Team in Anhui Province (2024AH010002), and 
Anhui Provincial Natural Science Foundation for Excellent Young Scholars (2408085Y003).}}

\author{Lele Liu, \quad Yi-Zheng Fan, \quad Yi Wang\thanks{Corresponding author: \texttt{wangy@ahu.edu.cn}}, \quad Wenyan Wang \\[2mm]
{\small \it School of Mathematical Sciences, Anhui University, Hefei 230601, P. R. China}}

\date{}

\begin{document}
\maketitle

\begin{abstract}
The spread of a real symmetric matrix is defined as the difference between its largest and smallest eigenvalue. 
The study of graph-related matrices has attracted considerable attention, leading to a substantial body of findings. 
In this paper, we investigate a general spread problem related to $A_{\alpha}$-matrix of graphs. 
The $A_{\alpha}$-matrix of a graph $G$, introduced by Nikiforov in 2017, is a convex combinations of its diagonal degree matrix $D(G)$ and adjacency matrix $A(G)$, 
defined as $A_{\alpha} (G) = \alpha D(G) + (1-\alpha) A(G)$. Let $\lambda_1^{(\alpha)} (G)$ 
and $\lambda_n^{(\alpha)} (G)$ denote the largest and smallest eigenvalues of $A_{\alpha} (G)$, respectively. 
We determined the unique graph that maximizes $\lambda^{(\alpha)}_1 (G) - \beta\cdot\lambda^{(\gamma)}_n (G)$
among all connected $n$-vertex graphs for sufficiently large $n$, where $0 \leq \alpha < 1$, $1/2\leq \gamma < 1$ and $0<\beta\gamma\leq 1$. 
As an application, we confirm a conjecture proposed by Lin, Miao, and Guo [Linear Algebra Appl. 606 (2020) 1--22]. In addition, 
one of main results in [SIAM J. Discrete Math. 38 (2024) 590--608] is a simple corollary of our result by choosing $\alpha = \gamma = 1/2$ and $\beta = 1$.
\par\vspace{2mm}

\noindent{\bfseries Keywords:} $A_{\alpha}$-matrix, $A_{\alpha}$-eigenvalue, spread
\par\vspace{2mm}

\noindent{\bfseries AMS Classification:} 05C35; 05C50; 15A18
\end{abstract}

\section{Introduction}

% One of the central problems in spectral graph theory is to uncover the combinatorial 
% properties of a graph that are revealed through the algebraic properties of its 
% associated matrices, such as the adjacency matrix, (signless) Laplacian matrix, 
% distance matrix, etc. For an introduction to this topic we refer the reader to \cite{Cvetkovic-Rowlinson-Simic2010}.
 
Consider an $n \times n$ complex matrix $M$. The \emph{spread} of $M$ is defined as 
the maximum distance between any two of its eigenvalues. This concept 
was first introduced by Mirsky \cite{Mirsky1956} in 1956, with further significant 
results appearing in \cite{ML}. Over the years, the spread has garnered considerable 
interest from researchers; see, for instance, \cite{Deutsch1978,Johnson1985,NT,TRC}.

The spread of a matrix has also drawn interest in specific cases, particularly in the 
context of graphs. In 2001, Gregory, Hershkowitz, and Kirkland \cite{Gregory-Hershkowitz-Kirkland2001}
initiated the study of the spread of graphs and established several bounds on this quantity.
They also conjectured that among all $n$-vertex graphs, the graph that maximizes the spread 
of its adjacency matrix is $K_{\lfloor2n/3\rfloor} \vee \lceil n/3\rceil K_1$, which 
is the join of the clique on $\lfloor2n/3\rfloor$ vertices and an independent set on 
$\lceil n/3\rceil$ vertices. Subsequently, many scholars contributed to this conjecture 
and related problems of such extremal flavor for a fixed family of graphs, such as 
trees \cite{FXWL}, unicyclic graphs \cite{WS}, bicyclic graphs \cite{WZS}. We would 
like to mention that Breen, Riasanovsky, Tait, and Urschel \cite{BRTJ} confirmed this 
conjecture by using techniques from the theory of graph limits and numerical analysis. For the spread of the signless Laplacian matrix of graphs, commonly referred 
to as the $Q$-spread, Liu et al. \cite{Liu-Liu2010} and Oliveira et al. \cite{Oliveira-Lima-Abreu-Kirkland2010} 
presented some upper and lower bounds. The unique graph with maximum $Q$-spread was 
determined in \cite{Liu-Liu2010} and \cite{Oliveira-Lima-Abreu-Kirkland2010}, which 
is a union of a complete graph and an isolated vertex. Among connected graphs, the 
graph with the minimum $Q$-spread was independently determined by Das \cite{Das2012} 
and Fan and Fallat \cite{Fan-Fallat2012} using different proof techniques. Recently, 
Liu \cite{Liu2024} determined the unique $n$-vertex connected graph attaining the 
maximum $Q$-spread for sufficiently large $n$. This graph, known as the \emph{kite graph} 
$Ki_{n, n-1}$, is obtained by attaching a pendant edge to one vertex of a clique on $n-1$ vertices.

Let us begin with some definitions. Let $G$ be a simple undirected graph of order $n$. 
Nikiforov \cite{Nikiforov2017} introduced the $A_{\alpha}$-matrix of $G$, defined as 
a convex combinations of $D(G)$ and $A(G)$:
\[
A_{\alpha} (G) := \alpha D(G) + (1-\alpha) A(G),
\]
where $D(G)$ is the diagonal matrix whose entries are the degrees of the vertices of $G$,
$A(G)$ is the adjacency matrix of $G$, and $\alpha\in [0,1]$ is a real number. 
Obviously, $A_0(G) = A(G)$, and $2A_{1/2}(G)$ is exactly the signless Laplacian matrix $Q(G)$. 
In this sense, researches on the family $A_{\alpha} (G)$ can underpin a unified theory 
of both $A(G)$ and $Q(G)$. Let $\lambda_1^{(\alpha)} (G)$ and $\lambda_n^{(\alpha)} (G)$ 
denote the largest and smallest eigenvalues of $A_{\alpha} (G)$, respectively. 
Specially, $\lambda_1^{(0)}(G) = \lambda_1(G)$ is the largest eigenvalue of $A(G)$. 
Moreover, $2\lambda_1^{(1/2)}(G) =: q_1(G)$ and $2\lambda_n^{(1/2)}(G) =: q_n(G)$ are the 
largest and smallest eigenvalues of $Q(G)$, respectively.

Generalizing the notion of spread for $A(G)$ and $Q(G)$, Lin, Miao, and Guo \cite{LinMiaoGuo2020-1} 
studied the $A_{\alpha}$-spread $S_{\alpha} (G)$ of a graph $G$, which is defined as 
\[
S_{\alpha} (G) := \lambda_1^{(\alpha)} (G) - \lambda_n^{(\alpha)} (G).
\]
In \cite{LinMiaoGuo2020-1} and \cite{LinMiaoGuo2020-2}, the authors presented some lower and upper 
bounds on $A_{\alpha}$-spread of graphs. In particular, the following conjecture was proposed.

\begin{conjecture}[\cite{LinMiaoGuo2020-2}]\label{conj:alpha-spread}
Let $G$ be a connected graph with $n\geq 5$ vertices. If $1/2 \leq \alpha < 1$, then 
\[
S_{\alpha} (G) \leq S_{\alpha} (Ki_{n, n-1})
\]
with equality if and only if $G\cong Ki_{n, n-1}$.
\end{conjecture}

In this paper we confirm Conjecture \ref{conj:alpha-spread} for sufficiently large $n$ in a stronger form.

\begin{theorem}\label{thm:alpha-spread}
Let $0 \leq \alpha < 1$, $1/2\leq \gamma < 1$ and $0<\beta\gamma\leq 1$. Among all connected 
graphs on $n$ vertices, $\lambda^{(\alpha)}_1 (G) - \beta\cdot\lambda^{(\gamma)}_n (G)$ is 
maximized only by $Ki_{n,n-1}$ for sufficiently large $n$.
\end{theorem}

Theorem \ref{thm:alpha-spread} implies that Conjecture \ref{conj:alpha-spread} holds for 
sufficiently large $n$ by letting $\beta = 1$ and $\gamma = \alpha$. Additionally,
Theorem \ref{thm:alpha-spread} also contributes partially to a problem posed by
Breen, Riasanovsky, Tait, and Urschel \cite{BRTJ}, who asked for the maximization of
the function $\beta\lambda_1^{(0)}(G) - (1-\beta) \lambda_n^{(0)}(G)$ over a specified 
class of graphs, where $\beta\in [0,1]$. On the other hand, by choosing 
$\alpha = \gamma = 1/2$ and $\beta = 1$, we immediately derive the following result.

\begin{corollary}[\cite{Liu2024}]
Among all connected graphs on $n$ vertices, $q_1(G) - q_n(G)$ is maximized only 
by $Ki_{n,n-1}$ for sufficiently large $n$.  
\end{corollary}

Finally, we can establish a new finding by setting $\alpha = 0$, $\beta = 2$ and 
$\gamma = 1/2$ in Theorem \ref{thm:alpha-spread}.

\begin{corollary}
Among all connected graphs on $n$ vertices, $\lambda_1(G) - q_n(G)$
is maximized only by $Ki_{n,n-1}$ for sufficiently large $n$.
\end{corollary}

\section{Preliminaries}
\label{sec2}

In this section we review some facts on eigenvalue\,--\,eigenvector equation, 
quadratic form of matrices and necessary conclusions that will be used throughout the paper.

\subsection{Terminology and notation}

Let $G$ be a graph. 
For subset $X \subset V(G)$, the subgraph of $G$ induced 
by $X$ is denoted by $G[X]$, and the graph obtained from $G$ by deleting $X$ is denoted 
by $G\setminus X$. For subsets $X, Y\subset V(G)$, we write $E(X, Y)$ for the set of edges 
with one vertex in $X$ and the other in $Y$, and write $E(X)$ for the set of edges induced 
by $X$. For a vertex $v$ of $G$, we write $d_G(v)$ and $N_G(v)$ for the degree of $v$ and 
the set of neighbors of $v$ in $G$, respectively. If the underlying graph $G$ is clear 
from the context, simply write $d(v)$ and $N(v)$. We denote the maximum degree and minimum 
degree of $G$ by $\Delta(G)$ and $\delta(G)$, respectively. For graph notation and 
terminology undefined here, we refer the reader to \cite{Bondy-Murty2008}.

\subsection{Eigenvalue-eigenvector equation and quadratic form}

Let $G$ be a graph of order $n$, and let $\bm{x}\in\mathbb{R}^n$ be an eigenvector of $A_{\alpha}(G)$ 
corresponding to an eigenvalue $\lambda$. The eigenvalue\,--\,eigenvector equation for 
$\lambda$ with respect to $u\in V(G)$ is 
\begin{equation}\label{eq:eigenvalue-vector-equation}
\lambda x_u = \alpha d(u) x_u + (1 - \alpha) \sum_{v\in N(u)} x_v.   
\end{equation}
Clearly, the quadratic form
$\langle A_{\alpha} (G)\bm{x}, \bm{x}\rangle$ can be represented as
\begin{equation}\label{eq:quadratic-form-for-Aalpha-matrix}
\langle A_{\alpha} (G)\bm{x}, \bm{x}\rangle 
= \sum_{uv\in E(G)} \big( \alpha(x_u-x_v)^2+2x_ux_v).
\end{equation}
Additionally, the quadratic form $\langle A_{\alpha} (G)\bm{x}, \bm{x}\rangle$ can be 
equivalently written as
\begin{equation}\label{eq:quadratic-form-for-Aalpha-matrix-2}
\langle A_{\alpha} (G)\bm{x}, \bm{x}\rangle 
= (2\alpha - 1) \sum_{u\in V(G)} d(u) x_u^2 + (1-\alpha) \sum_{uv\in E(G)} (x_u + x_v)^2.
\end{equation}

\subsection{Basic properties of $A_{\alpha}(G)$}

Since $A_{\alpha} (G)$ is a real symmetric matrix, Rayleigh's principle implies 
the following assertion.

\begin{proposition}[\cite{Nikiforov2017}]\label{prop:Rayleigh-principle}
Let $\alpha\in [0,1]$ and $G$ be a graph of order $n$. Then 
\[
\lambda_1^{(\alpha)} (G) = \max_{\|\bm{x}\|_2=1} \langle A_{\alpha} (G)\bm{x}, \bm{x}\rangle, \quad 
\lambda_n^{(\alpha)} (G) = \min_{\|\bm{x}\|_2=1} \langle A_{\alpha} (G)\bm{x}, \bm{x}\rangle.
\]
Moreover, if $\bm{x}$ is a unit vector, then $\lambda_1^{(\alpha)} (G) = \langle A_{\alpha} (G)\bm{x}, \bm{x}\rangle$ 
if and only if $\bm{x}$ is an eigenvector to $\lambda_1^{(\alpha)} (G)$ and 
$\lambda_n^{(\alpha)} (G) = \langle A_{\alpha} (G)\bm{x}, \bm{x}\rangle$ if and only 
if $\bm{x}$ is an eigenvector to $\lambda_n^{(\alpha)} (G)$.
\end{proposition}

\begin{proposition}[\cite{Nikiforov2017}]\label{prop:Aalpha-properties}
Let $0\leq\alpha < 1$. The following conclusions hold.
\begin{enumerate}
\item[$(1)$] If $\alpha \geq 1/2$, then $A_{\alpha} (G)$ is positive semidefinite.

\item[$(2)$] $\lambda_n^{(\alpha)} (G) < \alpha\cdot\delta (G)$.
\end{enumerate}
\end{proposition}

For $0\leq\alpha\leq 1$, $\beta > 0$ and $0\leq\gamma\leq 1$, as stated in 
Proposition \ref{prop:Rayleigh-principle}, the expression 
$\lambda_1^{(\alpha)} (G) - \beta\cdot\lambda_n^{(\gamma)} (G)$ 
can be written as 
\begin{equation}\label{eq:max-min-Aalpha-1-n}
\lambda_1^{(\alpha)} (G) {-} \beta\cdot\lambda_n^{(\gamma)} (G) 
= \max_{\|\bm{x}\|_2=\|\bm{z}\|_2=1} 
\sum_{uv\in E(G)} \!\!\big( \alpha (x_u {-} x_v)^2 {-} \beta\gamma (z_u {-} z_v)^2 {+} 2(x_ux_v {-} \beta z_uz_v) \big).
\end{equation}

The subsequent statement extends the inequality established by Hong, Shu, and Fang \cite{Hong-Shu-Fang2001}.

\begin{lemma}[\cite{HuangLinXue2020}]\label{lem:upper-bound-for-Aalpha}
Suppose that $G$ is an $n$-vertex graph with $m$ edges. Let $\delta$ and $\Delta$ be the 
minimum and maximum degree of $G$. If $\alpha\in [0, 1)$, then
\[
\lambda_1^{(\alpha)} (G) \leq \frac{ (1 {-} \alpha)(\delta {-} 1) {+} 
\sqrt{(1 {-} \alpha)^2 (\delta {-} 1)^2 + 4(\alpha\Delta^2 {+} (1 {-} \alpha) (2m {-} (n {-} 1)\delta))} }{2}.
\]
\end{lemma}

\section{Proof of Theorem \ref{thm:alpha-spread}}
\label{sec4}

Throughout this section, we always assume that $\alpha\in [0,1)$, $\gamma \in [1/2, 1)$, $\beta\gamma\in (0,1]$, 
and that $G$ is a graph that maximizes $\lambda_1^{(\alpha)}(G) - \beta \cdot \lambda_n^{(\gamma)}(G)$ 
among all connected $n$-vertex graphs. Let $\bm{x}$ be a nonnegative unit eigenvectors of $A_{\alpha}(G)$ 
corresponding to $\lambda_1^{(\alpha)}(G)$, and let $\bm{z}$ be a unit eigenvectors of $A_{\gamma}(G)$ 
corresponding to $\lambda_n^{(\gamma)}(G)$. We also set 
$\lambda_1^{(\alpha)} := \lambda_1^{(\alpha)}(G)$ and $\lambda_n^{(\gamma)} := \lambda_n^{(\gamma)}(G)$ for short.

Let us begin with a few basic properties that will be used repeatedly.

\subsection{Preliminary lemmas}

\begin{lemma}\label{lem:Aalpha-lower-bound}
$\lambda_1^{(\alpha)} > n - 2 - \beta\gamma$ and $0\leq \lambda_n^{(\gamma)} < 2/\beta$.
\end{lemma}

\begin{proof}
For the kite graph $Ki_{n,n-1}$, according to Proposition \ref{prop:Aalpha-properties}(2), $\lambda_n^{(\gamma)} (Ki_{n,n-1}) < \gamma$.
In light of the maximality of $\lambda_1^{(\alpha)} - \beta\cdot \lambda_n^{(\gamma)}$, we conclude that 
\begin{align*}
\lambda_1^{(\alpha)} - \beta\cdot \lambda_n^{(\gamma)}
& \geq \lambda_1^{(\alpha)} (Ki_{n,n-1}) - \beta \cdot \lambda_n^{(\gamma)}(Ki_{n,n-1}) \\
& > \lambda_1^{(\alpha)} (K_{n-1}) - \beta\gamma \\
& = n - 2 - \beta\gamma.
\end{align*}
Since $\lambda_n^{(\gamma)} \geq 0$ for $\gamma\in [1/2, 1)$, we have 
$\lambda_1^{(\alpha)} \geq \lambda_1^{(\alpha)} - \beta\cdot \lambda_n^{(\gamma)} > n - 2 - \beta\gamma$. 

On the other hand, it follows from $\lambda_1^{(\alpha)} \leq \Delta (G) \leq n-1$ that 
\[ 
n - 2 - \beta\gamma < \lambda_1^{(\alpha)} - \beta\cdot \lambda_n^{(\gamma)} \leq (n-1) - \beta\cdot \lambda_n^{(\gamma)},
\] 
which yields that $\lambda_n^{(\gamma)} < (1+\beta\gamma)/\beta \leq 2/\beta$. This completes the proof of the lemma.
\end{proof}

Lemma \ref{lem:Aalpha-lower-bound} implies that $\lambda_1^{(\alpha)} > n-2-\beta$. By employing Lemma \ref{lem:upper-bound-for-Aalpha}, 
we arrive at the subsequent assertion.

\begin{lemma}\label{lem:size-lower-bound-for-Aalpha}
$2e(G) > n^2 - \frac{ 3-\alpha+2\beta }{1-\alpha} n$.
\end{lemma}

\begin{proof}
By Lemma \ref{lem:upper-bound-for-Aalpha} and $\Delta (G) \leq n-1$ we see $\lambda_1^{(\alpha)} \leq f(\delta (G))/2$, 
where $f(x)$ is given by
\[
f(x) := (1 {-} \alpha)(x {-} 1) {+} \sqrt{(1 {-} \alpha)^2 (x {-} 1)^2 {+} 4\big[ \alpha (n {-} 1)^2 {+} (1 {-} \alpha) (2e(G) {-} (n {-} 1)x) \big]}.
\]
It is claimed that $f(x)$ is non-increasing in $x$. Indeed, we have
\begin{align*}
\frac{f'(x)}{1-\alpha} 
& = 1 - \frac{ 2(n-1) - (1-\alpha)(x-1) }{ \sqrt{(1 {-} \alpha)^2 (x {-} 1)^2 {+} 4\big[ \alpha (n {-} 1)^2 {+} (1 {-} \alpha) (2 e(G) {-} (n {-} 1)x) \big]} } \\
& = \frac{f(x) - 2(n-1)}{\sqrt{(1 {-} \alpha)^2 (x {-} 1)^2 {+} 4\big[ \alpha (n {-} 1)^2 {+} (1 {-} \alpha) (2 e(G) {-} (n {-} 1)x) \big]}} \\
& \leq 0,
\end{align*}
the last inequality is due to $e(G) \leq \binom{n}{2}$. It follows that 
\[
\lambda_1^{(\alpha)} \leq \frac{f(1)}{2}
= \sqrt{\alpha (n-1)^2 + (1-\alpha) (2e(G)-n+1) }.
\]
On the other hand, as shown in Lemma \ref{lem:Aalpha-lower-bound}, 
$\lambda_1^{(\alpha)} > n - 2 - \beta$. Therefore, 
\[
\sqrt{\alpha (n-1)^2 + (1 - \alpha) (2e(G) - n + 1) } > n - 2 - \beta.
\]
By rearranging, we obtain
\begin{align*}
2(1-\alpha)\cdot e(G) 
& > (1-\alpha) n^2 - (3-\alpha+2\beta) n + \beta^2 + 4\beta + 3 \\
& > (1-\alpha) n^2 - (3-\alpha+2\beta) n.
\end{align*}
This completes the proof.
\end{proof}

\begin{lemma}\label{lem:adjacent-iff-for-Aalpha}
For any vertices $u$ and $v$, if $\alpha (x_u - x_v)^2 - \beta\gamma (z_u - z_v)^2 + 2 (x_ux_v - \beta z_uz_v ) > 0$, 
then $u$ and $v$ are adjacent; if $\alpha (x_u - x_v)^2 - \beta\gamma (z_u - z_v)^2 + 2 (x_ux_v - \beta z_uz_v ) < 0$ 
and $G-uv$ is connected, then $u$ and $v$ are non-adjacent.
\end{lemma}

\begin{proof}
According to \eqref{eq:max-min-Aalpha-1-n}, if $\alpha (x_u - x_v)^2 - \beta\gamma (z_u - z_v)^2 + 2 (x_ux_v - \beta z_uz_v ) > 0$ 
and $u$, $v$ are non-adjacent, then
\begin{align*}
& ~\big(\lambda_1^{(\alpha)} (G+uv) -\beta\cdot\lambda_n^{(\gamma)} (G+uv)\big) 
- \big(\lambda_1^{(\alpha)} (G) -\beta\cdot\lambda_n^{(\gamma)} (G)\big) \\
\geq & ~\alpha (x_u - x_v)^2 - \beta\gamma (z_u - z_v)^2 + 2 (x_ux_v - \beta z_uz_v ) > 0,
\end{align*}
a contradiction which implies that $u$ and $v$ are adjacent. Likewise, we can show that if 
$\alpha (x_u - x_v)^2 - \beta\gamma (z_u - z_v)^2 + 2 (x_ux_v - \beta z_uz_v ) < 0$ and $G-uv$ is 
connected, then $u$ and $v$ are non-adjacent. 
\end{proof}

\subsection{Understanding the structure}

Fix a sufficiently small constant $\varepsilon > 0$, whose value will be chosen later. Let
\[
S:=\Big\{v\in V(G): |z_v| < \frac{\varepsilon}{\sqrt{n}}\Big\}, \hspace{3mm}
T:= \Big\{v\in V(G): x_v < \frac{1}{2\sqrt{n}}\Big\},
\]
and $L:= V(G)\setminus S$. We further partition $L$ into two subsets:
\[
B:= \Big\{v\in L: z_v > 0\Big\}, \hspace{3mm}
C:= \Big\{v\in L: z_v < 0\Big\}.
\]

We first prove that $L$ has bounded size.

\begin{lemma}\label{lem:size-L}
$|L| = o(n)$.
\end{lemma}

\begin{proof}
We assume by contradiction that $|L| = \Omega (n)$.
To finish the proof we shall prove $\lambda_n^{(\gamma)} = \Omega (n)$, and therefore 
get a contradiction to $\lambda_n^{(\gamma)} < (1 + \beta\gamma)/\beta$. 

To this end, note that $L = B \cup C$. Hence, at least one of $B$ and $C$ has size $\Omega(n)$. 
If $|B| = \Omega (n)$, we have $|E(B)| = \Omega(n^2)$. Otherwise,
if $|E(B)| = o(n^2)$, then 
\[
e(G) < \binom{n}{2} - \bigg(\binom{|B|}{2} - o(n^2)\bigg). 
\]
This is a contradiction to Lemma \ref{lem:size-lower-bound-for-Aalpha}. 
Likewise, if $|C| = \Omega (n)$, we also have $|E(C)| = \Omega(n^2)$.
Hence, $|E(B)\cup E(C)| = \Omega (n^2)$.

Since $\bm{z}$ is a unit eigenvector corresponding to $\lambda_n^{(\gamma)}$, 
by \eqref{eq:quadratic-form-for-Aalpha-matrix-2} we see
\begin{align*}
\lambda_n^{(\gamma)}
& = (2\gamma - 1) \sum_{u\in V(G)} d(u) z_u^2 + (1 - \gamma) \sum_{uv\in E(G)} (z_u + z_v)^2 \\
& \geq (1 - \gamma) \sum_{uv\in E(G)} (z_u + z_v)^2 \\
& = (1 - \gamma) \sum_{uv\in E(G)} (z_u^2 + z_v^2) + 2(1 - \gamma) \sum_{uv\in E(G)} z_u z_v \\
& \geq 2(1 - \gamma) \sum_{uv\in E(G)} (|z_u z_v| + z_u z_v) \\
& = 2(1 - \gamma) \sum_{uv\in E(G),\, z_u z_v>0} (|z_u z_v| + z_u z_v) \\
& \geq 2(1 - \gamma) \sum_{uv\in E(B)\cup E(C)} (|z_u z_v| + z_u z_v).
\end{align*}
Recall that $|z_v| \geq \frac{\varepsilon}{\sqrt{n}}$ for each $v\in B\cup C$, we have 
$\lambda_n^{(\gamma)} = \Omega (n)$, a contradiction completing the proof.
\end{proof}

Applying Lemma \ref{lem:size-L}, we will demonstrate that $G$ necessarily contains a vertex of small degree.

\begin{lemma}\label{lem:at-least-one-vertex-o(n)-for-Aalpha}
The graph $G$ has at least one vertex with degree $o(n)$.
\end{lemma}

\begin{proof}
Assume by contradiction that $d(v) = \Omega (n)$ for each $v\in V(G)$.
%{\color{red}We first show that $|z_v| = O(n^{-1/2})$ for every $v\in V(G)$.} Indeed, consider any 
%vertex $v$ of $G$, the eigenvalue\,--\,eigenvector equation \eqref{eq:eigenvalue-vector-equation} 
%for $\lambda_n^{(\gamma)}$ with respect to $v$ implies that
%\[
%|\big( \lambda_n^{(\gamma)} - \gamma d(v) \big) z_v| \leq (1 - \gamma) \sum_{u\in N(v)} |z_u| < (1 - \gamma) \sqrt{n},
%\]
%where the last inequality follows from Cauchy--Schwarz inequality. 
%Recall that $d(v) = \Omega (n)$, and $\lambda_n^{(\gamma)} = O(1)$ by Lemma \ref{lem:Aalpha-lower-bound}. 
%We have $|(\lambda_n^{(\gamma)} - \gamma d(v)) z_v| = \Omega (n) \cdot |z_v|$,
%which implies $|z_v| = O(n^{-1/2})$. 
Let $v_0$ be a vertex such that $|z_{v_0}| = \max\{|z_w|: w\in V(G)\}$. Without loss of generality, we can assume 
$z_{v_0}\leq -n^{-1/2}$. Consider the following eigenvalue\,--\,eigenvector equation with respect to $v_0$,
\[
0 \leq \big( \lambda_n^{(\gamma)} - \gamma d(v_0) \big) z_{v_0} = (1 - \gamma) \sum_{u\in N(v_0)} z_u.
\]
Since $\lambda_n^{(\gamma)} = O(1)$, $d(v_0)=\Omega(n)$ and $z_{v_0}\leq -n^{-1/2}$, we have 
\[
\sum_{u\in N(v_0)} z_u = \frac{ \lambda_n^{(\gamma)} - \gamma d(v_0) }{ 1 - \gamma }\cdot z_{v_0} = \Omega(\sqrt{n}).
\]
On the other hand, by Cauchy--Schwarz inequality and the fact $\|\bm{z}\|_2 = 1$, we find that
\[
\sum_{u\in N(v_0)} z_u 
\leq \sum_{u\in N(v_0)\, \cap L} |z_u| + \sum_{u\in N(v_0) \setminus L} |z_u| 
\leq \sqrt{|L|} + \varepsilon\sqrt{n},
\]
which yields that $|L| = \Omega(n)$ for sufficiently small $\varepsilon$. This is a contradiction
to Lemma \ref{lem:size-L}.
\end{proof}

Next, we will show that $G$ has exactly one vertex of small degree. Before continuing,
we need the following results.

\begin{lemma}\label{lem:x-bound-for-Aalpha}
For each $v\in V(G)$, $x_v < \frac{1}{\sqrt{n}} + O(n^{-3/2})$.
\end{lemma}

\begin{proof}
By the eigenvalue\,--\,eigenvector equation and Cauchy--Schwarz inequality we have
\[
\big( \lambda_1^{(\alpha)} - \alpha d(v) \big) x_v = (1 - \alpha) \sum_{u\in N(v)} x_u \leq (1 - \alpha) \sqrt{d(v)},
\]
which, together with Lemma \ref{lem:Aalpha-lower-bound} and $d(v) \leq n - 1$, gives 
\[
x_v \leq \frac{ (1 - \alpha) \sqrt{d(v)} }{ \lambda_1^{(\alpha)} - \alpha d(v) }
< \frac{ (1 - \alpha) \sqrt{n} }{ (1 - \alpha) n - 4 } = \frac{1}{\sqrt{n}} + O\Big(\frac{1}{n^{3/2}}\Big),
\]
as desired.
\end{proof}

\begin{lemma}\label{lem:T-size-for-Aalpha}
$|T| = O(1)$.
\end{lemma}

\begin{proof}
By Lemma \ref{lem:x-bound-for-Aalpha} and $\|\bm{x}\|_2 =1$, we deduce that
\begin{align*}
1 & = \sum_{u\in T} x_u^2 + \sum_{u\in V(G)\setminus T} x_u^2 \\
& < \frac{|T|}{4n} + (n - |T|) \bigg(\frac{1}{\sqrt{n}} + O\Big(\frac{1}{n^{3/2}}\Big)\bigg)^2 \\
& = 1 - \frac{3|T|}{4n} + O\Big(\frac{1}{n}\Big).
\end{align*}
Solving this inequality, we obtain the desired result.
\end{proof}

\begin{lemma}\label{lem:x-v-precise-value-for-Aalpha}
For each $v\in V(G)$ with $d(v) = \Omega (n)$, $x_v = \Theta(n^{-1/2})$.
\end{lemma}

\begin{proof}
Let $v$ be a vertex with $d(v) = \Omega (n)$. Using the eigenvalue\,--\,eigenvector 
equation with respect to $v$, we conclude that
\[
( \lambda_1^{(\alpha)} - \alpha d(v) )x_v 
\geq (1-\alpha) \sum_{u\in N(v)\setminus T} x_u \geq (1-\alpha) \cdot \frac{d(v) - O(1)}{2\sqrt{n}} 
= \Omega(\sqrt{n}).
\]
The second inequality follows from Lemma \ref{lem:T-size-for-Aalpha}. On the other hand, 
$\lambda_1^{(\alpha)} < n$. Therefore, $x_v = \Omega(n^{-1/2})$. Combining with 
Lemma \ref{lem:x-bound-for-Aalpha}, we get the desired result.
\end{proof}

\begin{lemma}\label{lem:z-v-precise-value-for-Aalpha}
For each $v\in V(G)$ with $d(v) = \Omega (n)$, $|z_v| = o(n^{-1/2})$.
\end{lemma}

\begin{proof}
Suppose on the contrary that there exists a vertex $v$ with $d(v) = \Omega (n)$ such that 
$|z_v|> c/\sqrt{n}$ for some $c>0$. By the eigenvalue\,--\,eigenvector equation for 
$\lambda_n^{(\gamma)}$ with respect to $v$, we have 
\begin{align*}
\bigg|\sum_{u\in N(v)} z_u \bigg| 
= \frac{\big|\big( \lambda_n^{(\gamma)} - \gamma d(v) \big) z_v\big|}{1 - \gamma} 
= \Omega(\sqrt{n}).
\end{align*}
On the other hand, we see
\[
\bigg|\sum_{u\in N(v)} z_u \bigg|
\leq \sum_{u\in N(v) \cap L} |z_u| + \sum_{u\in N(v) \setminus L} |z_u| 
\leq \sqrt{|L|} + \varepsilon\sqrt{n},
\]
which yields that $|L| = \Omega(n)$ for sufficiently small $\varepsilon$, 
a contradiction completing the proof.
\end{proof}

We are now prepared to show that $G$ contains exactly one vertex of small degree.

\begin{lemma}\label{lem:unique-vetex-o(n)-for-Aalpha}
The graph $G$ has exactly one vertex with degree $o(n)$.
\end{lemma}

\begin{proof}
We prove this lemma by contradiction. Let $M$ be the set of vertices whose degrees are $o(n)$. Then
\[
2 e(G) < |M|\cdot o(n) + (n - |M|) n = n^2 - |M| (1-o(1)) n.
\]
Moreover, by Lemma \ref{lem:size-lower-bound-for-Aalpha}, we have $|M| = O(1)$. Assume that $|M|\geq 2$. 
In the following we shall get contradiction by considering the cases $|M|\geq 3$ and 
$|M| = 2$, respectively. Together with Lemma \ref{lem:at-least-one-vertex-o(n)-for-Aalpha} 
this implies that $|M| = 1$.

To finish the proof, set $p(u,v):= \alpha x_u^2 + 2(1-\alpha) x_ux_v + \alpha x_v^2$. By 
Lemma \ref{lem:x-bound-for-Aalpha}, $x_v < (1+o(1)) n^{-1/2}$ for each $v \in V(G)$. Together with \eqref{eq:quadratic-form-for-Aalpha-matrix} 
and Proposition \ref{prop:Rayleigh-principle}, this gives
\begin{align*}
\lambda_1^{(\alpha)} 
& = \sum_{uv\in E(G\setminus M)} p(u,v) + \sum_{uv\in E(M, V(G)\setminus M) \, \cup \, E(M)} p(u,v) \\ 
& < \lambda_1^{(\alpha)}(K_{n-|M|}) + 2(1+o(1)) \bigg( \sum_{u\in M} d(u) \bigg) \cdot \frac{1}{n} \\
& < n-|M|-1 + o(1),
\end{align*}
where the last equality is due to $|M| = O(1)$. If $|M|\geq 3$, then $\lambda_1^{(\alpha)} < n-4+o(1)$, 
a contradiction to Lemma \ref{lem:Aalpha-lower-bound}. 

Now, assume that $|M|=2$. Let $M = \{u, v\}$, and $|z_u| \geq |z_v|$. We shall show that $\lambda_n^{(\gamma)}$ 
is not excessively small, specifically, $\lambda_n^{(\gamma)} > \gamma^2 (1 - \gamma)/12$.
Equation \eqref{eq:quadratic-form-for-Aalpha-matrix-2} implies that
\begin{equation}\label{eq:eq-1}
\begin{split}
\lambda_n^{(\gamma)} 
& = (2\gamma - 1) \sum_{w\in V(G)} d(w) z_w^2 + (1-\gamma) \sum_{ij\in E(G)} (z_i + z_j)^2 \\
& \geq (1-\gamma) \sum_{ij\in E(G)} (z_i + z_j)^2.
\end{split}
\end{equation}
On the other hand, in view of Lemma \ref{lem:z-v-precise-value-for-Aalpha} we find
\[
z_u^2 + z_v^2 = 1 - \sum_{s\in V(G)\setminus\{u,v\}} z_s^2 = 1- o(1),
\]
which yields that 
\begin{equation}\label{eq:lower-bound-zu}
|z_u| > \frac{1}{\sqrt{2}} - o(1).
\end{equation}
We consider the following two cases.

{\bfseries Case 1.} $d(u)\geq 2$. Let $w$ be a neighbour of $u$ other than $v$. 
By Lemma \ref{lem:z-v-precise-value-for-Aalpha}, $|z_w| = o(n^{-1/2})$. Recall that $\frac{1}{2} \leq \gamma < 1 $.
Hence, \eqref{eq:eq-1} and \eqref{eq:lower-bound-zu} together imply  
\[
\lambda_n^{(\gamma)} \geq (1-\gamma) (z_u + z_w)^2 > \frac{\gamma^2 (1 - \gamma)}{12}.
\]

{\bfseries Case 2.} $d(u) = 1$. It is enough to consider the case that $u$
and $v$ are adjacent, otherwise we have $\lambda_n^{(\gamma)} > \gamma^2 (1-\gamma)/12$ 
using similar arguments as Case 1. Now, let $uv\in E(G)$. By the eigenvalue\,--\,eigenvector 
equation for $\lambda_n^{(\gamma)}$ with respect to $u$,
\begin{equation}\label{eq:relation-zu-zv}
(\lambda_n^{(\gamma)} - \gamma) z_u = (1 - \gamma) z_v.
\end{equation}
Obviously, we may assume that $|z_v| > c$ for some small constant $c>0$, otherwise, 
we have $\lambda_n^{(\gamma)} > \gamma - o(1)$ by \eqref{eq:relation-zu-zv}.
Let $w$ be a neighbor of $v$ other than $u$. It follows from Lemma \ref{lem:z-v-precise-value-for-Aalpha} that
\begin{equation}\label{eq:gamma-n-lower-bound}
\lambda_n^{(\gamma)} \geq (1 - \gamma) (z_v + z_w)^2 > \frac{1 - \gamma}{4} z_v^2.
\end{equation}
From \eqref{eq:relation-zu-zv}, we know $z_v = \frac{(\lambda_n^{(\gamma)} - \gamma) z_u}{1 - \gamma}$.
Combining this with \eqref{eq:gamma-n-lower-bound} yields
\[
\lambda_n^{(\gamma)} > \frac{1 - \gamma}{4} z_v^2 > \frac{(\lambda_n^{(\gamma)} - \gamma)^2}{4(1-\gamma)} z_u^2 
> \frac{(\lambda_n^{(\gamma)} - \gamma)^2}{12(1-\gamma)},
\]
where the last inequality follows from \eqref{eq:lower-bound-zu}. The above inequality is equivalent to 
\[
(\lambda_n^{(\gamma)})^2 - (12 - 10\gamma) \lambda_n^{(\gamma)} + \gamma^2 < 0.
\]
Solving this inequality, we obtain 
\begin{align*}
\lambda_n^{(\gamma)} 
& > \frac{\gamma^2}{6 - 5\gamma + \sqrt{(6-5\gamma)^2 - \gamma^2}} \\
& > \frac{\gamma^2}{2(6 - 5\gamma)} \\
& > \frac{\gamma^2 (1-\gamma)}{12}.
\end{align*}
To sum up, we have $\lambda_n^{(\gamma)} > \gamma^2 (1-\gamma)/12$. 

Finally, combining with Lemma \ref{lem:Aalpha-lower-bound}, we find that 
\begin{align*}
\lambda_1^{(\alpha)} 
& = \big( \lambda_1^{(\alpha)} - \beta\cdot\lambda_n^{(\gamma)} \big) +  \beta\cdot\lambda_n^{(\gamma)} \\
& > n - 2 - \beta\gamma +  \frac{\beta\gamma^2 (1-\gamma)}{12} \\
& \geq n - 3 + \frac{\beta\gamma^2 (1-\gamma)}{12}.
\end{align*}
This is a contradiction with $\lambda_1^{(\alpha)} < n - |M| - 1 + o(1) = n-3+o(1)$. 
The proof of the lemma is completed.
\end{proof}

Hereafter, we assume that $w$ is the unique vertex such that $d(w) = o(n)$. 

\subsection{Proof of Theorem \ref{thm:alpha-spread}}

We now combine the results from the previous subsection to prove Theorem \ref{thm:alpha-spread}.

\begin{proof}[Proof of Theorem \ref{thm:alpha-spread}.]
For any $u,v\in V(G)\setminus \{w\}$, it follows from Lemma \ref{lem:x-v-precise-value-for-Aalpha} 
and Lemma \ref{lem:z-v-precise-value-for-Aalpha} that
\[ 
\alpha (x_u - x_v)^2 - \beta\gamma (z_u - z_v)^2 + 2 (x_ux_v - \beta z_uz_v ) > 0.
\] 
Combining with Lemma \ref{lem:adjacent-iff-for-Aalpha}, we deduce that the induced subgraph 
$G[V(G)\setminus\{w\}]$ is a complete graph. To finish the proof, we show that $d(w)=1$. Indeed,
by Lemma \ref{lem:z-v-precise-value-for-Aalpha} again, 
\[ 
z_w^2 = 1 - \sum_{u\in V(G)\setminus \{w\}} z_u^2 = 1 - o(1).
\]
If $d(w) \geq 2$, then deleting one edge between $w$ and $V(G)\setminus \{w\}$ will increase 
$\lambda_1^{(\alpha)} (G) - \beta\cdot \lambda_n^{(\gamma)} (G)$ by Lemma \ref{lem:adjacent-iff-for-Aalpha}, 
a contradiction completing the proof. 
\end{proof}

\end{document}